\documentclass[submission,copyright,creativecommons]{eptcs}
\providecommand{\event}{ACT 2026} % Name of the event you are submitting to

\usepackage{iftex}

\ifpdf
  \usepackage{underscore}         % Only needed if you use pdflatex.
  \usepackage[T1]{fontenc}        % Recommended with pdflatex
\else
  \usepackage{breakurl}           % Not needed if you use pdflatex only.
\fi

\usepackage{graphicx} % Required for inserting images
\usepackage{amsthm}
\usepackage{amsmath}
\usepackage{amssymb}
\usepackage{stmaryrd}
\usepackage{xcolor}
\usepackage{tcolorbox}
\usepackage{tabularx}
\usepackage{wrapfig}
\usepackage{physics}
\usepackage{bbold}
\theoremstyle{plain}
\newtheorem{theorem}{Theorem}
\newtheorem{lemma}[theorem]{Lemma}
\newtheorem{corollary}[theorem]{Corollary}
\newtheorem{proposition}[theorem]{Proposition}

\theoremstyle{definition}
\newtheorem{definition}[theorem]{Definition}
\newtheorem{example}[theorem]{Example}
\theoremstyle{remark}
\newtheorem{remark}[theorem]{Remark}
\usepackage{enumerate}
\usepackage{tikz}
\usetikzlibrary{cd, automata, positioning, arrows, calc}
\usepackage{macros}

\title{Combinatorial manifolds and Kleene's theorem, homotopically}
\author{Yorgo Chamoun
\institute{LIX, Ecole polytechnique\\ Palaiseau, France}
\email{yorgo.chamoun@polytechnique.edu}
}
\def\titlerunning{Combinatorial manifolds and Kleene's theorem}
\def\authorrunning{Y. Chamoun}
\begin{document}
\maketitle

\begin{abstract}
We give a general method to build categories of combinatorial manifolds, i.e. categories of combinatorial objects satisfying some local property at every ``point'', as coreflective subcategories of categories of relational presheaves. To do this, we crucially rely on unique factorization systems, and we can interpet our technique as a way of building a model category whose cofibrant objects are exactly the combinatorial manifolds. We then illustrate the usefulness of this point of view by two applications. First we build a category of euclidean precubical sets, i.e. precubical sets that locally look like a grid (of some fixed dimension), and show that it is coreflective in the category of relational precubical sets. This is the combinatorial analog of eulidean locally ordered spaces and the blowup construction from directed topology. Secondly, we show how to give an abstract proof of Kleene's theorem from automata theory by defining ``manifold automata'' that behave well with respect to concatenation.
\end{abstract}

\section{Introduction}
Relational presheaves~\cite{rosenthal1991free} over some category $\ccat$ are lax functors $\ccat^{op}\to\Rel$ with $\Rel$ the category of sets and relations. We can interpret these objects as ``open'' combinatorial structures. For example, if $\ccat$ is the simplex category $\Delta$, for every $n\in\Delta$, there is a relational simplicial set $S_n$ such that $S_n(n)=\{*\}$ and $S_n(m)=\varnothing$ for $m\ne n$, which can be interpreted as the open $n$-simplex. This topological interpretation has been confirmed by recent work on their realization~\cite{chamoun2026realization}. Thus, using relational presheaves, we are able to define ``combinatorial manifolds'', i.e.~combinatorial objects which are defined by a local property, where local means ``on some open neighborhood of every point''. Since we are in the combinatorial world, and moreover in a very well-behaved category (the category of relational presheaves over any category is locally finitely presentable), we can expect this definition to be given by an elegant categorical formalism. This is indeed the case, and the goal of this paper is to show that this can be done via unique factorization systems. In some cases, we can even define a model structure where the ``manifolds'' are exactly the cofibrant objects, even though this point of view does not seem to be fundamentally richer, since the other factorization system of the model structure is trivial. This is developped in \S\ref{sec:general}, after a quick overview of relational presheaves. We then give two applications of this methodology. 

\paragraph{Euclidean precubical sets.}
A \emph{precubical set} is a presheaf on the cube category $\square$, and can be thought of as a collection of $n$-cubes for every dimension $n$ glued together on their faces. Allowing a precubical set to be relational just means that we allow a cube to have any number of faces (including zero) in every direction. Precubical sets are prototypical higher directed combinatorial objects, and can be used to model parallel programs in computer science for example, where an $n$-cube represents $n$ operations executed in parallel~\cite{fajstrup2016directed}. In this context, the subclass of \emph{euclidean} precubical sets, i.e.~precubical sets which are locally isomorphic to a (filled) $n$-grid for some fixed dimension $n$, is particularly interesting. From the computer scientist perspective, the programs represented by euclidean precubical sets are locally just $n$ independent parallel processes, and are thus easy to optimize with respect to their execution time, e.g.~using tools from  differential geometry~\cite{haucourt2025non}.

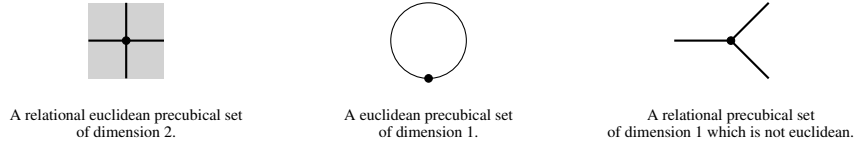
\[\begin{tikzpicture}
    \fill[color=mygray] (0,0) rectangle (1,1);
    \draw[thick] (0.5,0) -- (0.5,1);
    \draw[thick] (0,0.5) -- (1,0.5);
    \filldraw (0.5,0.5) circle (.05);
	\draw (0.5,-0.5) node{\tiny A relational euclidean precubical set};
	\draw (0.5,-0.7) node{\tiny of dimension 2.};

	\begin{scope}[xshift=40mm]
		\filldraw (0.5,0) circle (.05);
		\draw (0.5,0.5) circle (0.5);
		\draw (0.5,-0.5) node{\tiny A euclidean precubical set};
		\draw (0.5,-0.7) node{\tiny of dimension 1.};
	\end{scope}

	\begin{scope}[xshift=80mm]
		\draw (-0.25,0.5) edge[thick] (0.5,0.5);
		\filldraw (0.5,0.5) circle (.05);
		\draw (0.5,0.5) edge[thick] (1,1);
		\draw (0.5,0.5) edge[thick] (1,0);
		\draw (0.5,-0.5) node{\tiny A relational precubical set};
		\draw (0.5,-0.7) node{\tiny of dimension 1 which is not euclidean.};
	\end{scope}
  \end{tikzpicture}\]

\noindent In \S\ref{sec:pcsets}, motivated by this fact, we define a model structure on the category of relational precubical sets, and we show that its cofibrant replacement corresponds to 
a combinatorial \emph{blowup} in the vein of~\cite{haucourt2025non,chamoun2026realization}
but with a better universal property. In other words, we give an abstract definition of euclidean precubical sets, which comes with a canonical way of approximating a precubical set by a euclidean one. This example is a good illustration of our techniques, since the manifold interpretation is explicit.

\paragraph{Relational automata.}
An \emph{automaton} over some alphabet $\Sigma$ is a (possibly infinite) $\Sigma$-labelled graph with distinguished initial and accepting vertices. Vertices are called \emph{states} and edges are called \emph{transitions}. A \emph{word} $w=w_1 w_2\dots w_n$, i.e. a finite list of letters, is \emph{recognized} by an automaton $A$ if there is a $w$-labelled \emph{path} in $A$ from an initial to an accepting state, i.e. if there is a morphism  
\[\odot\hspace{-0.8ex}\xrightarrow{w_1}\hspace{-0.9ex}\circ\hspace{-0.9ex}\xrightarrow{w_2} \cdots\xrightarrow{w_n}\hspace{-0.9ex}\otimes \;\to \;A\]

\begin{wrapfigure}{r}{0.2\textwidth}
	\vspace{-5.5ex}
\[\hspace{-4.5ex}\begin{tikzpicture}[scale=0.45]
    \node (A) at (0,0) {$\odot$};
    
    % Nœud avec symbole ⊗
    \node (B) at (0,0) {$\otimes$};
    
    % Boucle sur le nœud A avec étiquette a
    \draw[->] (A) to [out=140, in=40, looseness=50, distance=3cm] node[midway, above] {a} (A);
	\draw[->] (A) to [out=-140, in=-40, looseness=50, distance=3cm] node[midway, below] {b} (A);
	\draw (0,-4) node{(A)};

	\begin{scope}[yshift=-100mm]
% Nœud avec symbole ⊙
    \node (A) at (-2,0) {$\odot$};
    
    % Nœud avec symbole ⊗
    \node (B) at (2,0) {$\otimes$};
    
    % Boucle sur le nœud A avec étiquette a
    \draw[->] (A) to [out=140, in=40, looseness=50, distance=3cm] node[midway, above] {a} (A);
    
    % Boucle sur le nœud B avec étiquette b
    \draw[->] (B) to [out=140, in=40, looseness=50, distance=3cm] node[midway, above] {b} (B);
    
    % Flèche epsilon entre A et B
    \draw[->] (A) -- node[above] {$\varepsilon$} (B);
	\draw (0,-1.5) node{(B)};
	\end{scope}
  \end{tikzpicture}
  \vspace{-3ex}\]
\end{wrapfigure}
\noindent (where the dotted state is initial and the crossed state is accepting).
The \emph{language} recognized by an automaton $A$ is the set of words recognized by $A$. \emph{Kleene's theorem}~\cite{sakarovitch2009elements,sipser1996introduction} states that the languages recognized by \emph{finite} automata are exactly the \emph{regular languages}, i.e.~the languages generated by the empty language and the singleton languages $\{a\}$ for every $a\in\Sigma$ under  union, concatenation ($L_1 * L_2:=\{w_1*w_2\,|\,w_1\in L_1,w_2\in L_2\}$, where concatenation of words is the ususal concatenation of finite lists) and Kleene star ($L^*:=\bigcup_{n\in\N}L^n$). To prove that every regular language is recognized by some finite automaton, one procedes by induction on the constructors of these languages. In particular, one should prove that given two automata recognizing languages $L_1$ and $L_2$ respectively, they can be concatenated into an automaton recognizing $L_1 * L_2$. This is not obvious, and the standard proof uses \emph{automata with silent transitions}, i.e.~transitions labelled by the empty word $\epsilon$, and shows that these are equivalent to ordinary automata. For example, consider the languages $\{a\}^*$ and $\{b\}^*$ recognized by the automata with a unique initial and accepting state and a loop labelled by $a$ (respectively $b$). If we naively concatenate the automata by glueing the accepting state $v$ of the first and the initial state $v'$ of the second, we get the automaton $(A)$ recognizing $\{a,b\}^*\ne\{a\}^* * \{b\}^*$. What we should do instead is adding an $\epsilon$-transition between $v$ and $v'$ $(B)$, and then ``eliminate'' the $\epsilon$. So automata with $\epsilon$-transition can be seen as a technical tool to handle concatenation of automata.
We suggest in this paper that \emph{relational automata} are a (categorically) better such technical tool.
Indeed, in \S\ref{sec:automata}, looking at the category of relational automata, we abstractly define the automata which behave well with respect to composition. This is a local property, since we only need to look at the neighborhood of initial and accepting states. As an application, we give a simple abstract proof of the corresponding direction of Kleene's theorem.

\section{Theoretical preliminaries}\label{sec:general}

\subsection{Relational presheaves}\label{sec:preliminaries}
\begin{definition}
  \label{relpsh}
  A \emph{relational presheaf} $P$ over some category $\ccat$ is a lax functor $\ccat^{op}\to \Rel$.
  A \emph{morphism} $\alpha:P\Rightarrow Q$ of relational presheaves over $\ccat$ is an oplax natural transformation.
\end{definition}
Relational presheaves over $\ccat$ form a category noted $\RelPsh(\ccat)$. In concrete terms, a relational presheaf~$P$ amounts to the data of a family of sets $P(c)$ indexed by objects $c\in\ccat$ and a family of relations $P(f)$ indexed by morphisms~$f$, with $P(f)\subseteq P(c)\times P(d)$ for $f:d\to c$, subject to (writing $a\to _f b$ for $(a,b)\in P(f)$):
\begin{enumerate}
	\item (reflexivity) $x\to_{\id_c} x$ for every $x\in P(c)$;
	\item (transitivity) if $a\to _f b$ and $b\to _g c$ then $a\to _{g\circ f} c$ for composable $f$ and $g$.
\end{enumerate}

Similarly, a morphism $\alpha:P\Rightarrow  Q$ amounts to a family of \emph{functions} $(\alpha_c:P(c)\to Q(c))_{c\in\ccat}$ such that $a\to _f b$ implies $\alpha_c(a)\to_f\alpha_d(b)$ for every $f:d\to c$. $\RelPsh(\ccat)$ is locally finitely presentable. The inclusion $\Set\to\Rel$ induces a functor $\Psh(\ccat)\to\RelPsh(\ccat)$, which has both left and right adjoints~\cite{chamoun2026realization}.

\begin{wrapfigure}{2}{0.1\textwidth}
	%\vspace{-2ex}
	\[
	\begin{tikzpicture}
		\draw[->] (0,0) -- (1,0);
		\filldraw (1.1,0.1) circle (0.05);
		\filldraw (1.1,-0.1) circle (0.05);
	\end{tikzpicture}
	\]
	\vspace{-5ex}
\end{wrapfigure}
\vspace{-2ex}
~ 

\begin{example}\label{ex:graphs}
	The category of directed graphs is a presheaf topos $\Psh(\gcat)$ where $\gcat$ is the category $0\rightrightarrows 1$ with two objects and two parallel morphisms. For a graph $G$, $G(0)$ is the set of \emph{vertices} and $G(1)$ is the set of \emph{edges}. A relational graph is a relational presheaf over $\gcat$. Intuitively, a relational graph is a graph where the edges may have zero, one or multiple sources and targets. For example, in the relational graph on the right, the unique edge has no source and two targets.
\end{example}

\subsection{The model structure}
In this section, we fix a category $\ccat$, 
which is supposed to be locally finitely presentable. If $I$ is a set of morphisms of $\ccat$, 
the class of morphisms of $\ccat$ with the right (resp.~left) lifting property with respect to $I$ 
is noted $I^\boxslash$ (resp.~$^\boxslash I$). 
We will prefer the point of view of model categories over that of factorization systems, essentially for vocabulary reasons. 
Recall that a model structure on $\ccat$ is a triple $(W,C,F)$ of classes of morphisms, respectively called \emph{weak equivalences}, \emph{cofibrations} and \emph{fibrations}, subject to some conditions~\cite[Definition 1.1.3]{hovey2007model}. Intuitively, $W$ is a class of morphisms that we want to invert, and $C$ and $F$ ensure that it goes well. In particular, $(W\cap C,F)$ and $(C,W\cap F)$ are (weak) factorization systems. Recall that we can build \emph{cofibrantly generated} model structures on $\ccat$,
in the following way:
\begin{theorem}[Theorem 2.1.19 of~\cite{hovey2007model}]
	Suppose given a subcategory $W$ and two sets $I$ and $J$ of morphisms of $\ccat$ 
	satisfying the following:
	\begin{enumerate}
		\item $W$ is closed under retract and satisfies the 2-out-of-3 property;
		\item $I^\boxslash\subseteq J^\boxslash \cap W$;
		\item $^\boxslash(J^\boxslash)\subseteq \,^\boxslash(I^\boxslash)\cap W$;
		\item (2) or (3) is an equality.
	\end{enumerate}
	Then there is a cofibrantly generated model structure on $\ccat$ 
	whose cofibrations are given by $^\boxslash(I^\boxslash)$, 
	fibrations by $J^\boxslash$ and weak eauivalences by $W$.
\end{theorem}

In this case, $I$ is the set of \emph{generating cofibrations}, 
and $J$ is the set of \emph{generating trivial cofibrations}. Now notice that in the precedent theorem, if $I$ is such that $I^\boxslash$ satisfies 2-out-of-3,
then we can just take $W:=I^\boxslash$ and $J:=\varnothing$, 
and we automatically get a model structure where all morphisms are fibrations.
Now we need an easy way to find such sets.
The idea is to look at sets $I$ such that the lifting property defining $I^\boxslash$ is in fact a \emph{unique} lifting property.

\begin{definition}
	Let $f:c\xrightarrow{}d$ be a morphism. 
	We define the \emph{codiagonal} $\nabla_f:d\sqcup_c d\xrightarrow{}d$ (where $d\sqcup_c d$ is the coequalizer of $(f,f):c\to d\sqcup d$) of $f$ to be the morphism $(\id_d,\id_d)$.
	
	% We form the following pushout:
	% \begin{center}
	% 	\begin{tikzcd}
	% 	c\arrow[r, "f"]\arrow[d, "f"'] & d\arrow[d]\\
	% 	d \arrow[r] & d\sqcup_c d
	% 	\end{tikzcd}
	% \end{center}
	% We write $\nabla_f:d\sqcup_c d\xrightarrow{}d$ for the the \emph{codiagonal} $(\id_d,\id_d)$ of $f$.
\end{definition}

\noindent The idea of adding codiagonals in order to force unique lifting properties goes back to~\cite{gabriel1971lokal}. The following lemma is folklore (see~\cite[\S4.5]{bousfield1977constructions} for example), and is proven in Appendix~\ref{sec:proofs}.

\begin{lemma}\label{lem:unique_lift}
	Let $i$, $f$ be morphisms. $f$ has the unique right lifting property with respect to $i$
	if and only if $f$ has the right lifting property with respect to $i$ and $\nabla_i$.
\end{lemma}

\begin{remark}\label{rem:nabla}
	A lift with respect to a codiagonal $\nabla_i$ is always unique,
	because $\nabla_{\nabla_i}$ is an isomorpism.
\end{remark}

\begin{definition}
	A set $I$ of morphisms is \emph{stable by codiagonal} if 
	for every $f\in I$, $\nabla_f$ is in $^\boxslash(I^\boxslash)$. 
	It is \emph{of cofibrant domain} if 
	for every $f\in I$, the unique morphism $0\xrightarrow{}\dom(f)$ from the initial object is in $^\boxslash(I^\boxslash)$. 
\end{definition}

\begin{proposition}\label{prop:2of3}
	Let $I$ be a set of morphisms which is stable by codiagonal and of cofibrant domain.
	Then $I^\boxslash$ satisfies 2-out-of-3.
\end{proposition}
\noindent\textit{Proof.} Since $(^\boxslash(I^\boxslash))^\boxslash=I^\boxslash$, we can suppose that $f\in I$ implies $\nabla_f\in I$ and $0\xrightarrow{}\dom(f)\in I$.

\begin{wrapfigure}{r}{0.2\textwidth}
  \vspace{-5mm}
  \[
  \begin{tikzcd}
		a \arrow[r, ""]\arrow[d, "i"']        & c \arrow[d, "g"]\\
		b \arrow[r, "t"] \arrow[ur, "h", dotted] & d \arrow[d, "f"]\\
		                                      & e
  \end{tikzcd}
  \]
  
  \[
  \begin{tikzcd}
		                                          & c \arrow[d, "g"]\\
		a \arrow[r, ""]\arrow[d, "i"']\arrow[ur, "h_1", dotted] & d \arrow[d, "f"]\\
		b \arrow[r, "t"] \arrow[uur, "h_2"', dotted]  & e
		\end{tikzcd}
  \]
  
\end{wrapfigure}
By Lemma~\ref{lem:unique_lift}, this means that $I^\boxslash$ is the class of morphisms which have the unique right lifting property with respect to $I$. Now take two composable morphisms in $\ccat$, $g:c\xrightarrow{}d$ and $f:d\xrightarrow{}e$, and a morphism $i$ of $I$. Suppose that $f\circ g$ is in $I^\boxslash$.  If $f$ is in $I^\boxslash$, we can consider a diagram as on the right. $h$ comes from the fact that $f\circ g\in I^\boxslash$. We have to check that $g\circ h = t$,  knowing that $f\circ g \circ h = f\circ t$, but this just comes from the uniqueness of the lift with respect to $f$. Since we can do the same replacing $i$ with $\nabla_i$, this lift is unique.

Now suppose that $g\in I^\boxslash$, and consider a diagram as on the right. $h_1$ comes from the fact that $0\xrightarrow{}a$ is in $I$, and $h_2$ from the fact that $f\circ g\in I^\boxslash$. Then the lift is given by $g\circ h_2$. It would be unique by the same argument as before, if we could prove that $0\xrightarrow{}\dom(\nabla_i)$ is in $^\boxslash(I^\boxslash)$. But this is true because it is the composite
\[0\xrightarrow{} a \xrightarrow{i} b \xrightarrow{} b\sqcup_a b\]
where the first and the second morphism are in $^\boxslash(I^\boxslash)$ by hypothesis, the third is in $^\boxslash(I^\boxslash)$ because it is the pushout of $i$ (along itself) which is in $I$.\qed

\begin{corollary}\label{cor:quillen}
	Let $I$ be a set of morphisms which is stable by codiagonal and of cofibrant domain.
	Then there is a Quillen model structure on $\ccat$ for which every morphism is fibrant,
	whose class of cofibrations is $^\boxslash(I^\boxslash)$, 
	and whose class of weak equivalences is $I^\boxslash$.
\end{corollary}

\noindent We allow ourselves to use the terminology ``cofibration'' for ``element of $^\boxslash(I^\boxslash)$'', even before proving that we have a model structure. Now we want to study the cofibrant replacement for this model structure. The following proposition is essentially in~\cite{gabriel1971lokal}. We do the proof in Appendix~\ref{sec:proofs}.

\begin{proposition}\label{prop:stable}
	Let $I$ be stable by codiagonal.
	Then $^\boxslash(I^\boxslash)$ is stable by codiagonal.
\end{proposition}

\begin{corollary}\label{cor:cofib}
	In the model structure of Corollary~\ref{cor:quillen}, 
	every cofibration has the unique left lifting property with respect to trivial fibrations
	(equivalently, weak equivalence).
	In particular, the cofibrant replacement is unique up to (unique) isomorphism,
	and a morphism is a weak equivalence if and only if it induces an isomorphism of the cofibrant replacements.
\end{corollary}

\begin{remark}\label{rem:cofib}
	Since the lift of cofibrations with respect to trivial fibrations is unique,
	it is easy to see that if $g\circ f$ and $f$ are cofibrations, 
	then $g$ is also a cofibration.
	In particular, any morphism between cofibrant objects is a cofibration.
\end{remark}

Our strategy to build categories of combinatorial manifolds is thus the following: identify a set of monomorphisms $I_+$, called \emph{positively generating cofibrations}, corresponding to the ways of locally building our combinatorial manifolds. Then define $I:=I_+\cup\{\nabla_i\,|\,i\in I_+\}$, and apply the construction of Corollary~\ref{cor:quillen}, to get a model structure with unique lifting property. This gives a coreflective subcategory, namely the category of cofibrant objects, which will exactly correspond to the expected manifolds. 

\begin{remark}
The model structure we obtain is not very interesting homotopically: it is one-dimensional in the sense of~\cite{balchin2019bousfield}, and all morphisms are fibrations. However, the conjunction of these two weaknesses is in fact interesting, since we can use the results in~\cite{balchin2019bousfield} to do a localization in a relatively easy way, by moreover keeping the same cofibrations, to get more interesting fibrations and trivial equivalences. 
%An example is shown below.
\end{remark}

\section{Precubical sets}\label{sec:pcsets}
As a proof of concept, we now show how to define a combinatorial analog of euclidean local orders, \textit{a.k.a} order manifolds. In this section, $\ccat$ is the category of relational precubical sets $\RelPsh(\square)$, as defined in~\cite{chamoun2026realization}. Recall that $\square$ is the free monoidal category with unit $0$ generated by $\gcat$ (Example~\ref{ex:graphs}). More concretely, the objects are natural numbers, and the morphisms are generated by $d^\eta_{n,i}:n\to n+1$ with $\eta\in\{-,+\}$, $n,i\in\N$ with $0\leq i\leq n$, subject to:
\[d^{\eta'}_{n+1,j}\circ d^\eta_{n,i}=d^\eta_{n+1,i}\circ d^{\eta'}_{n,j-1}\] 
for $0\leq i<j\leq n+1$. For a precubical set $P$, $P(n)$ is the set of \emph{$n$-dimensional cubes}, and the $P(f)$ are the \emph{face maps} (or \emph{face relations} in the relational case), where $\delta^\eta_{n,i}:=P(d^\eta_{n,i})$ maps an $n$-cube to its back ($\eta=-$) or front ($\eta=+$) face in direction $i$. Precubical sets are equipped with a tensor product defined by $P\otimes Q(n):=\bigsqcup_{i+j=n}P_i\times D_j$, with obvious face maps. More details can be found in~\cite[\S3.4.1]{fajstrup2016directed}.

\subsection{The model structure}
Fix a natural number $n$. We want to apply the construction of the previous section, so we just need to choose a suitable set $I_+$ of morphisms.

\begin{definition}
	Let $P$ be a relational precubical set, $c$ a cube of $P$.
	We define $\uparrow c$ the \emph{upward neighborhood} of $c$, as the relational precubical set whose cubes are pairs $(d,f)$ with $d\to_f c$ in $P$, and whose relations are given by $(d',g\circ f)\to_g(d,f)$ with $d'\to_g d$ in $P$.
\end{definition}
\noindent Let $V_0:=\,\cdot \xrightarrow{} \cdot$ the graph with one edge, $V_1:=\,\cdot \xrightarrow{}\cdot\xrightarrow{}\cdot$ the graph with two consecutive edges.

\begin{definition}
	\begin{enumerate}
		\item A \emph{closed euclidean brick} is a precubical set of the form $\bigotimes_{1\leq i\leq n}V_{\epsilon_i}$ with $\epsilon_i\in\{0,1\}$.
		We write $c(\epsilon)$ with $\epsilon=\epsilon_1\epsilon_2\dots\epsilon_n$ for the cube $(p_{\epsilon_1},\dots,p_{\epsilon_n})$ with $p_0$ the edge of $V_0$ and $p_1$ the central vertex of $V_1$.
		\item A \emph{euclidean brick} $B_\epsilon$ is a relational precubical set of the form $\uparrow c(\epsilon)$.
		\item The \emph{codimension} of a euclidean brick is the sum \(\epsilon_1+\cdots+\epsilon_n\). 
		\item A euclidean brick of codimension $k$ has a unique $(n-k)$-cube,
		which is of minimal dimension among the cubes of the brick, namely $c(\epsilon)$. 
		We call it the \emph{minimal cube} $\min(B_\epsilon)$ of the brick.
		\item The unique brick of codimension $0$ is the \emph{open $n$-cube}, and is noted $C_n$.
		\item A \emph{positively generating cofibration} is an inclusion of the form $i_\epsilon:B_\epsilon\setminus\{\min(B_\epsilon)\}\hookrightarrow B_\epsilon$.
	\end{enumerate}
\end{definition}

\begin{example}
	For $n=2$, the euclidean bricks are given by:
	\begin{align*}
  B_{00}
  &=
  \begin{tikzpicture}[baseline={([yshift=-.5ex]current bounding box.center)}]
    \fill[color=mygray] (0.25,0.25) rectangle (0.75,0.75);
  \end{tikzpicture}
  &
  B_{01}
  &=
  \begin{tikzpicture}[baseline={([yshift=-.5ex]current bounding box.center)}]
    \fill[color=mygray] (0,0) rectangle (0.5,1);
    \draw[thick] (0,0.5) -- (0.5,0.5);
  \end{tikzpicture}
  &
  B_{10}
  &=
  \begin{tikzpicture}[baseline={([yshift=-.5ex]current bounding box.center)}]
    \fill[color=mygray] (0,0) rectangle (1,0.5);
    \draw[thick] (0.5,0) -- (0.5,0.5);
  \end{tikzpicture}
  &
  B_{11}
  &=
  \begin{tikzpicture}[baseline={([yshift=-.5ex]current bounding box.center)}]
    \fill[color=mygray] (0,0) rectangle (1,1);
    \draw[thick] (0.5,0) -- (0.5,1);
    \draw[thick] (0,0.5) -- (1,0.5);
    \filldraw (0.5,0.5) circle (.05);
  \end{tikzpicture}
\end{align*}
where the edges are directed from left to right and from bottom to top, and the first (respectively second) dimension of a square is the horizontal (respectively vertical) one (this will be a convention for all the figures). The generating cofibration $i_{11}$ is given by the inclusion:
\[\begin{tikzpicture}
    \fill[color=mygray] (0,0) rectangle (1,1);
    \draw[thick] (0.5,0) -- (0.5,1);
    \draw[thick] (0,0.5) -- (1,0.5);
    \filldraw[white] (0.5,0.5) circle (.05);

	\begin{scope}[xshift=20mm]
		\draw (0,0.5) node {$\hookrightarrow$};
	\end{scope}

	\begin{scope}[xshift=30mm]
		\fill[color=mygray] (0,0) rectangle (1,1);
    \draw[thick] (0.5,0) -- (0.5,1);
    \draw[thick] (0,0.5) -- (1,0.5);
    \filldraw (0.5,0.5) circle (.05);
	\end{scope}
  \end{tikzpicture}\]
  The intuition for $i_{11}$ is the following: in the process of building a cofibrant object, if we want to add a vertex we need to make sure that the ``neighborhood'' of this vertex will look like a $2$-grid. We have similar intuitions for $i_{01}$ and $i_{10}$. For $i_{00}$, we can always (freely) add an open square, as it is always a (rather trivial) $2$-grid. As a consequence, one would expect that objects built out of $I$ are local (i.e.~at every cube) $2$-grids.
\end{example}

% \begin{wrapfigure}{r}{0.2\textwidth}
%   \vspace{-6mm}
%   \[
%   \begin{tikzcd}
% 		\varnothing \arrow[r]\arrow[d] & B_{\epsilon'}\arrow[d, "\iota_-"]\\
% 		B_{\epsilon'}\arrow[r, "\iota_+"'] & B_\epsilon
% 	\end{tikzcd}
% \]
% \vspace{-0.6cm}
% \end{wrapfigure}
Let $B_\epsilon$ be a euclidean brick.
Let $i$ such that $\epsilon_i = 1$.
There are two inclusions $\iota_-,\,\iota_+:B_{\epsilon'}\hookrightarrow B_\epsilon$,
where $\epsilon'$ is obtained from $\epsilon$ by replacing $\epsilon_i$ by $0$, 
which are induced by the two inclusions $V_0\hookrightarrow V_1$.
In addition, we have: 
\begin{equation}\label{eq:iota}
	\Im(\iota_-)\cap\Im(\iota_+)=\varnothing
\end{equation}
Writing $j:=|\{k\leq i\,|\, \epsilon_k=0\}|$ and $m:=|\{k\leq n\,|\, \epsilon_k=0\}|$, we also have that $\min(B_{\epsilon})=\delta_{m,j}^+(\iota_-(\min(B_{\epsilon'})))$ and $\min(B_\epsilon)=\delta_{m,j}^-(\iota_+(\min(B_{\epsilon'})))$. This just comes from the definition of the face relations in the tensor product. This last point is generalized at the beginning of \S\ref{sec:blowup}. See in particular Remark~\ref{rem:g_w}.

\begin{example}
	For the brick $B_{11}$ and $i=1$, we get:
	\[\begin{tikzpicture}
    \fill[color=mygray] (-.25,-0.5) rectangle (.25,0.5);
    \draw[thick] (-.25,0) -- (0.25,0);

	\begin{scope}[xshift=15mm]
		\draw (0,0) node {$\hookrightarrow$};
		\draw (0,-0.2) node {\tiny $\iota_-$};
	\end{scope}

	\begin{scope}[xshift=31mm]
		\fill[color=mygray] (-0.5,-0.5) rectangle (0.5,0.5);
    \draw[thick] (-0.5,0) -- (0.5,0);
    \draw[thick] (0,-0.5) -- (0,0.5);
    \filldraw (0,0) circle (.05);
	\end{scope}

	\begin{scope}[xshift=47mm]
		\draw (0,0) node {$\hookleftarrow$};
		\draw (0,-0.2) node {\tiny $\iota_+$};
	\end{scope}

	\begin{scope}[xshift=62mm]
	\fill[color=mygray] (-.25,-0.5) rectangle (.25,0.5);
    \draw[thick] (-.25,0) -- (0.25,0);
	\end{scope}
  \end{tikzpicture}\]
  Notice that $\delta^+_{0,0}(\min(\Im(\iota_-)))=\delta^-_{0,0}(\min(\Im(\iota_+)))=\min(B_{11})$.
\end{example}

More generally, we define a poset $D_\epsilon$ in the following way. Let $\{+,-,0,1\}^n$ be the poset defined by the pointwise order generated by $0<+,-<1$. Let $D_\epsilon$ be the coslice under $\epsilon$, i.e.~the subposet consisting of the $w\in\{+,-,0,1\}^n$ such that $w_i=0$ if and only if $\epsilon_i=0$. This poset has a maximal element $\epsilon$. We have a canonical projection $-\land \epsilon:\{+,-,0,1\}^n\to D_\epsilon$, and a morphism $p:\{+,-,0,1\}^n\to \{0,1\}^n$ sending $0,+,-$ to $0$ and $1$ to $1$. For every $w\in D_\epsilon$, we have an inclusion $\iota_w:B_{p(w)}\xrightarrow{}B_\epsilon$. For $w\leq w'$, we write $\iota_{w\leq w'}:=\iota_{w\land {p(w')}}:B_{p(w)}\to B_{p(w')}$. As a consequence of equation~\ref{eq:iota} above, if we have $w$ and $w'$ such that $w_i=+$ and $w'_i=-$ for some $i$, then the images of $\iota_w$ and $\iota_{w'}$ are disjoint. Otherwise, there is an element $w''\leq w,w'$ which is the biggest one with this property, and we have a pullback square 
\begin{equation}\label{diag:pb}
	\begin{tikzcd}
		B_{p(w'')} \arrow[dr,dotted,"\iota_{w''}"]\arrow[r,"\iota_{w''\leq w'}"]\arrow[d,"\iota_{w''\leq w}"'] & B_{p(w')}\arrow[d, "\iota_{w'}"]\\
		B_{p(w)}\arrow[r, "\iota_w"'] & B_\epsilon
	\end{tikzcd}
\end{equation} 

The set of positively generating cofibrations is noted $I_+$, and we define $I:=I_+\cup\{\nabla_i\,|\, i\in I_+\}$. $I$ is stable by codiagonal by definition (and Remark~\ref{rem:nabla}). 
We need to show that it is of cofibrant domain.
By the last part of the proof of Proposition~\ref{prop:2of3}, 
we just need to treat the case of positively generating cofibrations.

\begin{remark}\label{rem:Depsilon}
The poset $D_\epsilon$ can be thought of as the category of elements of the partial precubical set $B_\epsilon$: its elements represent the cubes of $B_\epsilon$, and we have $w\leq w'$ if and only if $w'$ is a face of $w$. The correspondence is given by $w\mapsto \iota_w(\min(B_{p(w)}))$. In particular, the maximal element $\epsilon$ represents the minimal cube, and indeed it is a face of every other cube of $B_\epsilon$. For example, for $B_{11}$ the correspondance is given by:
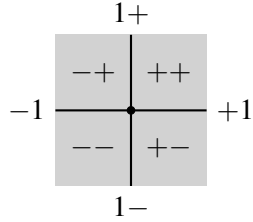
\[
  \begin{tikzpicture}
    \fill[color=mygray] (0,0) rectangle (2,2);
    \draw[thick] (1,0) -- (1,2);
    \draw[thick] (0,1) -- (2,1);
    \filldraw (1,1) circle (.05);
	\draw (0,1) node[left]{$-1$};
	\draw (1,2) node[above]{$1+$};
	\draw (2,1) node[right]{$+1$};
	\draw (1,0) node[below]{$1-$};
	\draw (0.5,0.5) node{$--$};
	\draw (1.5,0.5) node{$+-$};
	\draw (0.5,1.5) node{$-+$};
	\draw (1.5,1.5) node{$++$};
  \end{tikzpicture}
\]
\end{remark}

\begin{proposition}\label{prop:B-minB}
	$B_\epsilon\setminus\{\min(B_\epsilon)\}$ is cofibrant for every $\epsilon$.
\end{proposition}
\begin{proof}
	Since every morphism between cofibrant objects is a cofibration (Remark~\ref{rem:cofib}), we have that cofibrant objects are stable by colimits. So we procede by induction on the codimension of $B_\epsilon$. For codimension $0$, noting $\mathbf{0}$ the tuple only consisiting of zeros, $B_\mathbf{0}=C_n$ which has a unique cube, so $B_\mathbf{0}\setminus\{\min(B_\mathbf{0})\}=\varnothing$ the initial relational percubical set, which is cofibrant by definition. For the inductive step, suppose that the result was establish for codimension strictly smaller than $k>0$. As a consequence, every brick of codimension $<k$ is cofibrant. Now take a brick $B_\epsilon$ of codimension $k$. We define a diagram $F:D_\epsilon\to \ccat$ by: 
	\[F(w):=B_{p(w)}, \qquad F(w<w'):=\iota_{w\leq w'}\]
	We claim that $B_\epsilon\setminus\{\min(B_\epsilon)\}$ is the colimit of $F'$, the restriction of $F$ to $D_\epsilon\setminus\{\epsilon\}$. Indeed, since $F$ is a diagram, we have that $(F',B_\epsilon)$ is a cocone, so there is a morphism $\alpha:\colim F'\to B_\epsilon$. Clearly every cube of $B_\epsilon$ except the minimal one appears in $\Im(\alpha)$, so $\alpha$ factors through $B_\epsilon\setminus\{\min(B_\epsilon)\}$. Now by the explicit definition of colimits in relational presheaves given in~\cite[Proposition 36]{chamoun2026realization}, $\colim F'$ is a quotient of the disjoint union of the $B_{p(w)}$ by the equivalence relation generated by $\sim$, defined in the following way: given two cubes $c\in B_{p(w)}$ and $c\in B_{p(w')}$, we have $c\sim c'$ if and only if there is some $w''\leq w,w'$ and $c''\in B_{p(w'')}$ such that 
	\[\iota_{w''\leq w}(c'')=c \quad \text{and} \quad \iota_{w''\leq w'}(c'')=c'\]
	But again since $F$ is a diagram, $c\sim c'$ implies that $c$ and $c'$ represent the same cube of $B_\epsilon$ (i.e.~that $\iota_{w}(c)=\iota_{w'}(c')$). Conversely, if $c$ and $c'$ represent the same cube of $B_\epsilon$, by the pullback square (\ref{diag:pb}) and the discussion above it, $w$ and $w'$ are the same on every index where they are both in $\{+,-\}$, so we can find $w''\leq w,w'$, and a $c''$ witnessing that $c\sim c'$. So in fact $\sim$ is already an equivalence relation. Using this, it is straightforward to define an inverse $B_\epsilon\setminus\{\min(B_\epsilon)\}\to \colim F$ to $\alpha$, which concludes the induction.
\end{proof}

% \begin{wrapfigure}{r}{0.3\textwidth}
%   \begin{center}
%     \begin{tikzpicture}
%     \fill[color=mygray] (0,0) rectangle (2,2);
%     \draw[thick] (1,0) -- (1,2);
%     \draw[thick] (0,1) -- (2,1);
%     \filldraw (1,1) circle (.05);
% 	\draw (0,1) node[left]{$-1$};
% 	\draw (1,2) node[above]{$1+$};
% 	\draw (2,1) node[right]{$+1$};
% 	\draw (1,0) node[below]{$1-$};
% 	\draw (0.5,0.5) node{$--$};
% 	\draw (1.5,0.5) node{$+-$};
% 	\draw (0.5,1.5) node{$-+$};
% 	\draw (1.5,1.5) node{$++$};
%   \end{tikzpicture}
%   \end{center}
% \end{wrapfigure}

\subsection{Euclidean precubical sets}

As a consequence of the last section, the construction of \S\ref{sec:general} applies, and we get a model structure on $\ccat$. Now we study the cofibrant replacement in this model structure,
which is unique up to isomorphism (Proposition~\ref{prop:stable}).
We start by identifying a particular set of cofibrant objects, and we will later show that in fact every cofibrant object is in this set.

\begin{definition}
	A morphism of relational precubical sets $\alpha:P\to Q$ is \emph{a local embedding} if for all distinct cubes $a,\,b\in P$, if there is $w\in\square$ and $c\in P$ such that $a\to_w c$ and $b\to_w c$ then $\alpha(a)\ne\alpha(b)$.
\end{definition}

\begin{example}
	For example, the following morphism, that identifies $e$ and $e'$, is \emph{not} a local embedding, since $e$ and $e'$ have the same target.
	
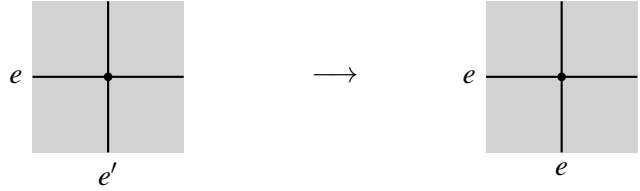
\[
  \begin{tikzpicture}
    \fill[color=mygray] (-1,-1) rectangle (1,1);
    \draw[thick] (0,-1) -- (0,1);
    \draw[thick] (-1,0) -- (1,0);
    \filldraw (0,0) circle (.05);
	\draw (-1,0) node[left]{$e$};
	\draw (0,-1) node[below]{$e'$};

	\begin{scope}[xshift=30mm]
		\draw (0,0) node{$\longrightarrow$}; 
		\draw (0.5,0.8) node{};
	\end{scope}

	\begin{scope}[xshift=60mm]
		\fill[color=mygray] (-1,-1) rectangle (1,1);
		\draw[thick] (0,-1) -- (0,1);
		\draw[thick] (-1,0) -- (1,0);
		\filldraw (0,0) circle (.05);
		\draw (-1,0) node[left]{$e$};
		\draw (0,-1) node[below]{$e$};
	\end{scope}
  \end{tikzpicture}
\]
\end{example}

\begin{definition}\label{def:euclidean_precubical_set}
	A relational precubical set $P$ is \emph{euclidean} (or \emph{$n$-euclidean}) if for every cube $c$ of $P$, there is a euclidean brick $B_\epsilon$ and a surjective local embedding $\phi_c:B_\epsilon\xrightarrow{}\,\uparrow c$, which is called a \emph{chart} at $c$.
\end{definition}

\begin{remark}
This is equivalent to the following more ``geometric'' definition:
the $2$-subdivision of $P$ is locally isomorphic to $B_\textbf{1}$, which is the combinatorial analog of the pointed ordered space $(\R^n,0)$. See~\cite[Appendix D]{chamoun2026realization}. 
\end{remark}

% \begin{remark}
% 	For dimension reasons, the chart $\phi_c$ of Definition \ref{def:euclidean_precubical_set} sends the minimal cube of $B_\epsilon$ to $c$. 
% 	In particular, if $k$ is the dimension of $c$, the codimension of $B_\epsilon$ is necessarily equal to $n-k$. 
% \end{remark}

\begin{proposition}\label{prop:eucl->cofib}
	Euclidean relational precubical sets are cofibrant.
\end{proposition}
\textit{Proof.}
	Let $E$ be euclidean. By induction on $k$, we show the following: the relational precubical subset $E_k$ of $E$ consisting of the cubes of dimension $> n-k$ is cofibrant. This will immediately prove the proposition. Since charts are surjective, $E_0=\varnothing$ so it is vacuously true.

	\begin{wrapfigure}{r}{0.35\textwidth}
		\vspace{-2ex}\begin{tikzcd}
		B_\epsilon\setminus\{\min(B_\epsilon)\} \arrow[r,"\phi'_c"]\arrow[d,"i_\epsilon"'] & E_{k}\arrow[d]\\
		B_\epsilon \arrow[r,"\phi_c"'] & E_{k}\cup\{c\}
	\end{tikzcd}
	\end{wrapfigure}
	Now suppose the result established for every integer strictly smaller than $k+1$ for some $0\leq k\leq n$. Let $c$ be a $(n-k)$-cube of $E$. By hypothesis, there is a euclidean brick $B_\epsilon$ of codimension $k$ and an chart $\phi_c:B_\epsilon\xrightarrow{}\,\uparrow c$ at $c$. $\phi_c$ induces a morphism $\phi_c:B_\epsilon\to E$ by composing with the first projection. This gives the commutative diagram on the right, where $\phi'_c$ is the restriction of $\phi_c$, which is easily seen to be cocartesian (intuitively, every relation between $c$ and other cubes of $E$ of higher dimension is encoded in $\phi_c$, by definition of $\uparrow c$). Since $E_{k+1}=E_{k}\cup \bigcup_{c\in E_{k+1}\setminus E_k}\{c\}$, we get that $E_{k}\to E_{k+1}$ is a cofibration, as a pushout of a sum of generating cofibrations. By the induction hypothesis, $\varnothing\to E_{k+1}$ is a cofibration.
	% Let $P$ be euclidean. 
	% First note that $P$ is of maximal dimension $n$.
	% Now take a trivial fibration $\beta:A\xrightarrow{}B$,
	% and a morphism $f:P\xrightarrow{}B$.
	% We want to define a morphism $\tilde{f}:P\xrightarrow{}B$.
	% For every $n$-cube $c$ of $P$, the following diagram has a unique lift:
	% \begin{center}
	% 	\begin{tikzcd}
	% 		& & A\arrow[d, "\beta"]\\
	% 		C_n \arrow[r, hook]\arrow[urr, dotted] & P \arrow[r,"f"'] & B
	% 	\end{tikzcd}
	% \end{center}
	% we define $\tilde{f}(c)$ to be the image of this lift.
	% Now for any cube $c$ of $P$, there is a brick $B_\epsilon$ and a local surjection $\phi_c$.
	% So we have a lift:
	% \begin{center}
	% 	\begin{tikzcd}
	% 		\bigsqcup_{1\leq i \leq 2^k}C_n \arrow[d, hook]\arrow[rrr, "\tilde{f}"] & & & A\arrow[d, "\beta"]\\
	% 		B_\epsilon \arrow[r, "\phi_c"']\arrow[urrr, dotted] & \uparrow c \arrow[r, hook] & P \arrow[r,"f"'] & B
	% 	\end{tikzcd}
	% \end{center}
	% since $\tilde{f}$ is already defined on the $n$-cubes.
	% Now $P$ is euclidean,
	% so we can show by induction on the dimension of the cubes of $B_\epsilon$ that
	% $h$ factors through $\phi_c$ by some $h'$, 
	% using the uniqueness of the lifts along positvely generating cofibrations.
	% We define $\tilde{f}(c):=h'(c)$.
	% Note that the lifts are unique,
	% so the different lifts for varying $c$ are compatible.
	% This finishes the definition of $\tilde{f}$. \\
	% \textcolor{red}{TODO: Write the induction. I fact, we should be able to avoid using uniqueness, by building P by successive pushouts.}
\qed

\subsection{The blowup}\label{sec:blowup}
The combinatorial blowup of a precubical set was defined in~\cite{chamoun2026realization} as an endofunctor of the category of relational precubical sets, satisfying that its realization in locally ordered spaces is exactly the blowup operation defined in~\cite{haucourt2025non,chamoun2025non}. But this is not completely satisfying: this operation is defined with respect to a topological one, and thus does not satisfy a good universal property in the combinatorial world. Here, we proceed in a different way: having defined the subcategory of euclidean precubical sets in the previous section, we prove now that it is coreflective. The proof of coreflectivity will follow from an abstract characterization of these objects as the cofibrant objects of the model structure defined above, illustrating the interest of such model structures. Note that the construction in not the same as in~\cite{chamoun2025non}, although it is also called \emph{blowup}. 

First, recall from~\cite{grandis2003cubical} that a morphism $f:m\to m+k$ in~$\square$ has a unique normal form
\[
  f=d^{\eta_k}_{m+k-1,i_k}\circ\ldots \circ d^{\eta_2}_{m+1,i_2}\circ d^{\eta_1}_{m,i_1}
\]
with $0\leq i_1<i_2<\ldots<i_k<m+k$, so it can be seen as a word $w\in\{+,-,0\}^{m+k}$ with exactly $k$ components different from $0$, 
where $w_{i_j}=\eta_j$ and $w_i=0$ for the other indices. Conversely, given a tuple $\epsilon\in\{0,1\}^n$ of codimension $n-m$ and $w\in D_\epsilon$ such that $|\{i\,|\,w_i\ne 0,1\}|=k$, we can form the word $\hat{w}$ by removing all the letters $1$ from $w$, and then, supposing that $\{i_1,\dots,i_k\}$ is the ordered set $\{i\,|\,\hat{w}_i\ne 0\}$, we can define $g_w:m\to m+k$ by 
\[
  g_w:=d^{\hat{w}_{i_k}^*}_{m+k-1,i_k}\circ\ldots \circ d^{\hat{w}_{i_1}^*}_{m+1,i_2}\circ d^{\hat{w}_{i_1}^*}_{m,i_1}
\]
with $+^*=-$ and $-^*=+$. 

\begin{remark}\label{rem:g_w}
With the same notations, one checks (recalling Remark~\ref{rem:Depsilon}) by induction on $k$ that $w\to_{g_w} \epsilon$ in $B_\epsilon$, and in fact that $\uparrow \epsilon$ is exactly $\{(w,g_w)\,|\,w\in D_\epsilon\}$.
\end{remark} 

\begin{definition}
	The \emph{blowup} (or \emph{$n$-blowup}) $\tilde{P}$ of a relational precubical set $P$ is (the relational precubical set) defined as follows:
	\begin{align*}
		& \tilde{P}(k) = \bigsqcup_{\epsilon\text{ of codimension }k}\Hom(B_\epsilon,P)\\
	    & \tilde{P}(g) = \{(f,f\circ \iota_w)\,|\, f:B_\epsilon\to P,\,w\in D_\epsilon,\,w\to_g \epsilon\}
	\end{align*}
	The \emph{blowup map} (or \emph{$n$-blowup map}) $\beta_P:\tilde{P}\xrightarrow{}P$ is defined by $\beta_P(f):= f(\min(\dom(f)))$.
\end{definition}

\begin{example}
	Consider the precubical set $X$ with one vertex $v$, one edge $e$ and one square $c$. For every euclidean brick $B_\epsilon$, there is a unique morphism $B_\epsilon\to X$. Since there is one euclidean brick of codimension $0$, one of codimension $2$, and two of codimension $1$, we get that the blowup of $X$ is the torus:
	\begin{center}
		\begin{tikzpicture}
			\filldraw[mygray] (0,0) rectangle (1,1);
			\draw (0.5,0.5) node{$c$};
			\filldraw (0,0) circle (0.05);
			\filldraw (0,1) circle (0.05);
			\filldraw (1,0) circle (0.05);
			\filldraw (1,1) circle (0.05);
			\draw (0,0) node[below left]{$v$};
			\draw (0,1) node[above left]{$v$};
			\draw (1,0) node[below right]{$v$};
			\draw (1,1) node[above right]{$v$};
			\draw (0,0) edge[thick, "$e_1$"'] (1,0);
			\draw (0,0) edge[thick, "$e_2$"] (0,1);
			\draw (0,1) edge[thick, "$e_1$"] (1,1);
			\draw (1,0) edge[thick, "$e_2$"'] (1,1);

			\begin{scope}[xshift=20mm]
				\draw (0.5,0.5) node{$\longrightarrow$}; 
				\draw (0.5,0.8) node{\tiny $\beta_X$};
			\end{scope}

			\begin{scope}[xshift=40mm]
				\filldraw[mygray] (0,0) rectangle (1,1);
			\draw (0.5,0.5) node{$c$};
			\filldraw (0,0) circle (0.05);
			\filldraw (0,1) circle (0.05);
			\filldraw (1,0) circle (0.05);
			\filldraw (1,1) circle (0.05);
			\draw (0,0) node[below left]{$v$};
			\draw (0,1) node[above left]{$v$};
			\draw (1,0) node[below right]{$v$};
			\draw (1,1) node[above right]{$v$};
			\draw (0,0) edge[thick, "$e$"'] (1,0);
			\draw (0,0) edge[thick, "$e$"] (0,1);
			\draw (0,1) edge[thick, "$e$"] (1,1);
			\draw (1,0) edge[thick, "$e$"'] (1,1);
			\end{scope}
		\end{tikzpicture}
	\end{center}
\end{example}

\begin{proposition}\label{prop:bup_eucl}
	The blowup of a relational precubical set is euclidean.
\end{proposition}
\begin{proof}
	Take a $k$-cube $f:B_\epsilon\xrightarrow{}P$ of $\tilde{P}$. 
	From the definition and Remark~\ref{rem:g_w}, we have an explicit description of $\uparrow f$ in $\tilde{P}$ as $\{(f\circ \iota_w,g_w)\,|\,w\in D_\epsilon\}$. We claim that $f$ itself is a chart at $f$. More precisely, recall (Remark~\ref{rem:Depsilon}) that the cubes of $B_\epsilon$ can be represented by the elements of $D_\epsilon$. So we can map every $w\in D_\epsilon$ to the cube $f\circ \iota_w$ of $\tilde{P}$ (in fact to $(f\circ \iota_w,g_w)$ in $\uparrow f$). This induces a morphism of relational precubical sets. We need to prove that this is a chart at $f$. We already argued that it is surjective on $\uparrow f$.
	We still need to check that it is a local embedding. First we prove the following: if we have $w,w'\in D_\epsilon$ such that $g_w=g_{w'}$, then either $w=w'$ or there is $i$ such that $w_i=1$ and $w'_i\in\{+,-\}$ (or the opposite). Indeed, following the construction of $g_w$, if we have $w_i\ne w'_i$ but they agree on every index $i$ such that $w_i=1$ or $w'_i=1$, then $g_w$ and $g_{w'}$ have different normal forms, so $g_w\ne g_{w'}$, which proves the claim.
	Now take $w,w'\in D_\epsilon$ and $g\in\square$ such that $w\to_{g} \epsilon$ and $w'\to_{g} \epsilon$. By Remark~\ref{rem:g_w}, this means that $g=g_w=g_{w'}$. By the claim, this implies that there exists $i$ such that $w_i=1$ and $w'_i\in\{+,-\}$ (or the opposite). But then $p(w)\ne p(w')$, so $(f\circ \iota_w,g_w)\ne (f\circ \iota_{w'},g_w)$. Clearly this case is enough, which concludes the proof.
\end{proof}

\begin{theorem}\label{thm:bup}
	Every blowup map is a trivial fibration.
\end{theorem}
\textit{Proof.}
	We prove this by induction on the codimension $k$ of $\epsilon$. For $k=0$, notice that there is a unique $n$-cube in $\tilde{P}$ over every $n$-cube of $P$, because there is no non-trivial endomorphism $C_n\to C_n$. 
	
	\begin{wrapfigure}{r}{0.3\textwidth}
		\vspace{-1ex}\begin{tikzcd}
		B_\epsilon\setminus\{\min(B_\epsilon)\} \arrow[r,"g"]\arrow[d,"i_\epsilon"'] & \tilde{P}\arrow[d,"\beta_P"]\\
		B_\epsilon \arrow[r,"f"']\arrow[ur,dotted,"h"'] & P
	\end{tikzcd}
	\end{wrapfigure}
	Now for the inductive step, take $f:B_\epsilon\to P$ with $\epsilon$ of codimension $k$ and a commutative square as on the right.
	By the induction hypothesis and the proof of Lemma~\ref{prop:B-minB}, there is a unique such $g$ making the square commute. In addition, by the very definition of $\tilde{P}$, there is a morphism $h:B_\epsilon\to \tilde{P}$ such that $\beta_P\circ h=f$, which is $w\mapsto f\circ\iota_w$. So by uniqueness $g$ is the restriction of this $h$ to $B_\epsilon\setminus\{\min(B_\epsilon)\}$, and $h$ is indeed a lift. We still need to prove uniqueness. In other words, we have to prove that for any lift $h$, we have $h(\epsilon)=f$. We have that $h(\epsilon)$ corresponds to a morphism $h(\epsilon):B_\epsilon\to P$ such that $h(\epsilon)(\epsilon)=f(\epsilon)$ (commutation of the lower triangle) and for $w\in B_\epsilon\setminus\{\min(B_\epsilon)\}$: 
	\begin{align*}
		h(\epsilon)(w) & = h(w)(\min(B_{p(w)})) \quad \text{(because $h(w)\to_{g_w} h(\epsilon$) in $\tilde{P}$)}\\
		               & = g(w)(\min(B_{p(w)})) \quad \text{(commutation of the upper triangle)}\\
					   & = f(w) \qquad\qquad\qquad\, \text{(uniqueness of $g$)}
	\end{align*}
	So indeed $h(\epsilon)=f$, which concludes. \qed
	%\vspace{-1ex}

As a consequence of Theorem~\ref{thm:bup} and Proposition~\ref{prop:bup_eucl}, 
the cofibrant replacement of $P$ is exactly $\tilde{P}$.
In addition, we deduce from Proposition~\ref{prop:bup_eucl} and Corollary~\ref{cor:cofib}
that a relational precubical set is cofibrant if and only if it is euclidean,
since it is cofibrant if and only if it is isomorphic to its cofibrant replacement.

\begin{theorem}\label{thm:cofib-euclidean}
	The cofibrant objects are exactly the euclidean relational precubical sets.
	The cofibrant replacement corresponds to the blowup.
\end{theorem}

\begin{corollary}
	Euclidean precubical sets are coreflective in relational precubical sets.
\end{corollary}

\noindent Note that this Corollary was absolutely not obvious from Definition~\ref{def:euclidean_precubical_set}, but is immediate from Definition/Theorem~\ref{thm:cofib-euclidean}, which illustrates the power of this point of view.

\section{Homotopy theory of automata}\label{sec:automata}
As another application, we prove Kleene's theorem using homotopical arguments, by defining a suitable notion of ``manifold automata'', ensuring that the concatenation of two such automata is well-behaved with respect to the language.
For a fixed alphabet $\Sigma$, the category of $\Sigma$-labelled graphs is a presheaf topos $\Psh(\int_\gcat \Sigma)$, where $\int_\gcat \Sigma$ is the category of elements of the graph $\Sigma$ with one vertex and $|\Sigma|$ transitions.

\begin{definition}
	A relational $\Sigma$-automaton is a relational $\Sigma$-labelled graph $G$ together with two distinguished sets of vertices $I_G$ and $T_G$ (representing initial and accepting states respectively). A morphism of relational automata $\alpha:A\to B$ is a morphism of relational labelled graphs preserving initial and accepting states, i.e. such that $\alpha(I_A)\subseteq I_B$ and $\alpha(T_A)\subseteq T_B$. The category of relational $\Sigma$-automata is noted $\Aut$. Non-relational automata form a subcategory $\sAut$.
\end{definition}

Note that $\Aut$ is locally finitely presentable, and the finitely presentable automata are exactly the finite ones. 

\begin{remark}
	In the figures below, $\odot$ represents an initial state, $\otimes$ an accepting state, $\circledast$ a state which is both initial and accepting, $\circ$ a state which is neither initial nor accepting, and $\to$ an edge.
\end{remark}

\begin{definition}
	The set of positively generating cofibrations $I_+$ is given by the morphisms
	\[i_\odot:\varnothing \hookrightarrow \odot, \quad 
	  i_\to:\varnothing \hookrightarrow \;\to, \quad 
	  i_s:\odot\to\;\hookrightarrow \odot\hspace{-0.8ex}\to, \quad 
	  i_\otimes:\;\to\;\hookrightarrow \;\to\hspace{-0.8ex}\otimes, \quad i_\circledast:\varnothing\hookrightarrow \circledast \]
	as well as 
	
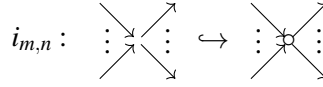
\[\begin{tikzpicture}
		\begin{scope}[xshift=-33mm]
			\draw (0,0) node {$i_{m,n}:$};
		\end{scope}

		\begin{scope}[xshift=-20mm]
		\draw[<-] (0.5,-0.5) -- (0.045,-0.045);
		\draw[<-](0.5,0.5) -- (0.05,0.05);
		\draw[->] (-0.5,0.5) -- (-0.05,0.05);
		\draw[->] (-0.5,-0.5) -- (-0.045,-0.045);
		\draw (-0.4,0.1) node {$\vdots$};
		\draw (0.4,0.1) node {$\vdots$};
		\end{scope}

		\begin{scope}[xshift=-10mm]
			\draw (0,0) node {$\hookrightarrow$};
		\end{scope}

		\draw (0,0) node {$\circ$};
		\draw[<-] (0.5,-0.5) -- (0.045,-0.045);
		\draw[<-](0.5,0.5) -- (0.05,0.05);
		\draw[->] (-0.5,0.5) -- (-0.05,0.05);
		\draw[->] (-0.5,-0.5) -- (-0.045,-0.045);	
		\draw (-0.4,0.1) node {$\vdots$};
		\draw (0.4,0.1) node {$\vdots$};	
	\end{tikzpicture}\]
for any numbers $m,\,n>0 $ of entering and leaving edges respectively. 
\end{definition}
Intuitively, in order for an automaton to behave well with respect to concatenation, initial states should only have edges going out, and accepting states only edges going in. In order to preserve these conditions while allowing to add edges between internal states, one should only add such a state \emph{after} adding all its entering and leaving transitions. Note that we did not label the edges in the previous definition. We should take all the possibilities, so the morphisms given above are in fact families of morphisms. For example, the morphism $\varnothing \hookrightarrow \;\xrightarrow{a}$ for $a\in\Sigma$ will be noted $i_\to(a)$. If we do not specify a label, we mean any possible labeling. Now we can define $I:=I_+\cup\{\nabla_i\,|\,i\in I_+\}$ as before. By Corollary~\ref{cor:quillen}, this defines a model structure on $\Aut$. Before proving Kleene's theorem, we need some intermediate results, which are all proven abstractly.

\begin{proposition}\label{prop:conditions}
	Any cofibrant automata satisfies:
	\begin{enumerate}
		\item Every initial state has only edges going out of it.
		\item Every accepting state which is not initial has only edges entering it.
	\end{enumerate}
\end{proposition}
\begin{proof}
	Easy induction, knowing that any cofibration is a retract of a countable composition of pushouts of sums of elements of $I$ (by the small object argument~\cite[\S 2.1.2]{hovey2007model}). For retracts, notice that retract inclusions preserve \emph{and reflect} initial and accepting states. For transfinite composition, we just use that the category is locally finitely presentable, so for example if $0\to X_\infty:=\colim_{i\in\N}X_i$ is a cofibration with $X_i$ cofibrant and satisfying the properties for all $i\in \N$, if there is an initial state in $X_\infty$ which has an edge going in, we get a morphism $\to\hspace{-0.8ex}\odot\;\hookrightarrow X_\infty\;$
	which then factors through some $X_i$, which is impossible since the $X_i$ are supposed to satisfy the properties. Lastly, the only possibly problematic operation is the pushout, but we just see that for any automata satisfying the conditions, the pushout along any element of $I$ preserves these conditions. 
\end{proof}

\begin{proposition}\label{prop:paths}
	The path automata $\odot\hspace{-0.8ex}\to\hspace{-0.9ex}\circ\hspace{-0.9ex}\to \cdots\to\hspace{-0.9ex}\otimes$ are cofibrant, for any length $n\geq 0$.
\end{proposition}

\begin{remark}\label{rem:simplify}
	In fact, the properties suggested by Propositions~\ref{prop:conditions} and~\ref{prop:paths} are the only properties needed to prove Kleene's theorem, as we will see in the proof below. So we could have avoided adding $i_{m,n}$ for $m,\,n>1$, but adding them gives a more satisfying cofibrant replacement, in the sense that it leaves the ``internal structure'' (i.e. the states that are not initial nor accepting) of the automaton essentially untouched. However, if we just take $i_{1,1}$ in $I_+$, we get the \emph{blowup model structure} on $\Aut$, which is also interesting since the cofibrant replacement then essentially corresponds to ``unfolding one step'' of the loops of the automaton. Going into the details would take us too far, so we leave this for future work. 
\end{remark}

The right adjoint to the inclusion $\Psh(\int_\gcat \Sigma)\to\mathbf{RelPsh}(\int_\gcat \Sigma)$ can be lifted to a right adjoint $\deter$ to the inclusion $\sAut\to\Aut$ (the set of vertices is unchanged by the right adjoint, so we just take the same initial and accepting states). Clearly applying the right adjoint to a relational automata does not change the language, by the adjoint property and the fact that the paths automata are non-relational.

\begin{lemma}\label{lem:replacement}
	Let $A\in\sAut$. There is a non-relational $\Sigma$-automaton $\normal(A)$ satisfying the conditions of Proposition~\ref{prop:conditions}, recognizing the same language, and having a unique initial state.
\end{lemma}
\begin{proof}
	Take the cofibrant replacement $\tilde{A}$ of $A$, then identify all its initial states, and apply $\deter$. The cofibrant replacement preserves the language because of Proposition~\ref{prop:paths}, and the conditions ensure that identifying initial states preserves the language. Finally, $\deter$ preserves the conditions because they can be stated in terms of morphisms into the automaton. We give details for the second condition: take a relational automaton $B$ satisfying that every accepting state with edges going out is also initial. Now take an accepting state $v$ of $\deter(B)$, and suppose that it has an edge going out. We get a morphism $o:\otimes\hspace{-0.8ex}\to\hspace{-0.9ex}\circ\;\to \det(B)$. Composing with the counit of the adjunction, we get a morphism $o':\otimes\hspace{-0.8ex}\to\hspace{-0.9ex}\circ\;\to B$, so by hypothesis it can be factored through $i:\otimes\hspace{-0.8ex}\to\hspace{-0.9ex}\circ\;\to\;\circledast\hspace{-0.8ex}\to\hspace{-0.9ex}\circ$ as a morphism $o'':\circledast\hspace{-0.8ex}\to\hspace{-0.9ex}\circ\;\to B$. Now by adjunction, this lifts to a morphism $o''':\circledast\hspace{-0.8ex}\to\hspace{-0.9ex}\circ\;\to \det(B)$, and by naturality we have $o=o'''\circ i$, which means that $v$ is initial. 
\end{proof}

\noindent This can be expressed as a homotopy pushout (followed by $\deter$). Indeed, take an automaton $A$, with initial states $I_A$. Now consider the homotopy pushout of $\sqcup_{I_A}\odot\hookrightarrow A$ along $i_A:\sqcup_{I_A}\odot\to\odot$. Since $\odot$ is cofibrant, one way of computing it is the following (see~\cite[Corollary 2.3.28]{cisinski2019higher}). First take a cofibrant replacement $\tilde{A}$ of $A$. Since $\sqcup_{I_A}\odot$ is cofibrant, $i_A$ lifts to a morphism $i'_A:\sqcup_{I_A}\odot\to \tilde{A}$. Finally, take the ordinary pushout of $i'_A$ along $\sqcup_{I_A}\odot\to\odot$. This is exactly the procedure described above, which is very satisfying: the ``good'' way of identifying the initial states of an automaton is to identify them \emph{homotopically}.

\begin{lemma}\label{lem:finite}
	The cofibrant replacement of a finite automaton is finite.
\end{lemma}
\begin{proof}
	Since trivial fibrations satisfy the unique lifting property with respect to $i_\to$, the cofibrant replacement preserves the number of edges. For vertices, first we notice that, by construction, a state in a cofibrant automaton is either: (1) initial; (2) accepting and the target of an edge; or (3) both the target and the source of some edges. Take a finite automaton $X$, and suppose that its cofibrant replacement has infinitely many vertices. For at least one vertex $v$ of $X$, looking at the fiber over $v$, at least one of these three families of vertices is infinite. But this contradicts the uniqueness of the lift with respect to $i_\odot$, $i_\otimes$ or $i_{m,n}$ (with $m$ the number of edges going in $v$, $n$ the number of edges going out, if both are $>0$).
\end{proof}

\noindent Finally, we prove the desired theorem.

\begin{theorem}
	Every regular language is recognized by some finite automaton.
\end{theorem}
\begin{proof}
	The empty language and the singleton languages clearly do. The union of languages is just given by the coproduct of the corresponding automata. We still have concatenation and Kleene star.
	Take two finite automata $A$ and $B$ recognizing languages $L_A$ and $L_B$ respectively. $\normal(A)$ and $\normal(B)$ are finite by Lemma~\ref{lem:finite}. Let $i$ be the initial state and $(x_j)_{j\in J}$ be the non-initial accepting states of $\normal(A)$, $v$ the initial state of $\normal(B)$. Let $A'$ be the automaton $\normal(A)$ where we forget that the $x_j$ and eventually $i$ are accepting, $B'$ be the automaton $\normal(B)$ where we forget that $v$ is initial (it can still be accepting). The conditions of Proposition~\ref{prop:conditions} ensure that the automaton $A'\sqcup_{v=x_j}B'$ recognizes $(L_A\setminus\{\epsilon\})*L_B$, and the automaton $\normal(A)_{/i =x_j}$ recognizes $(L_A)^*$. Note that if $\epsilon\in L_A$ then $L_A*L_B=((L_A\setminus\{\epsilon\})*L_B)\cup L_B$, so if $i$ is accepting, the coproduct of $A'\sqcup_{v=x_j}B'$ with $B$ is an automaton recognizing $L_A*L_B$.
\end{proof}

\section{Conclusion}
This paper shows how to build categories of combinatorial manifolds, i.e.~combinatorial objects which locally look like some fixed (set of) model(s). We argued that this can easily be obtained by combining relational presheaves and unique factorization systems. In addition to the abstract definition in terms of cofibrant objects, our method automatically implies the existence of a coreflector 
from combinatorial objects to combinatorial manifolds. With this in mind, we showed how to define the most obvious example of such combinatorial manifolds: precubical sets which are locally euclidean. Indeed, precubical sets are typical combinatorial objects which can be used to model topological spaces, and the euclidean condition is probably the most common local condition that one can impose on ``manifolds'' (except that here, the euclidean structure is the order structure, not the differential one). 

However, this example is not a mere illustration, and can also have concrete applications in concurrent programming. Indeed, our blowup construction produces a euclidean precubical set which, under some conditions which need to be explicited, faithfully represents the parallel program that was represented by the original precubical set. This kind of operations is known to be interesting, and our construction generalizes the one of~\cite{haucourt2025non}, which was only defined on graphs, to all precubical sets. 

On the other hand, we also showed that the concept of ``combinatorial manifold'' as defined here can be distant from any topological interpretation, and be useful in contexts like automata theory. Although there is no novelty in reproving Kleene's theorem, the fact that our proof is completely abstract suggests that this technique might generalize well to other flavours of automata. For example, it is interesting to compare our \S\ref{sec:automata} to~\cite[\S6-7]{bazille2026variants}, where the authors prove a Kleene theorem for relational higher dimensional automata. The difference is that we never have to ``guess'' how to transform an automaton into an equivalent one which behaves well with respect to concatenation; we just state what that means, locally, and then impose these conditions by a categorical procedure. It is essentially philosophical, but we believe that our method can give simple proofs in settings where guessing the correct answer would be challenging. 

Finally, we did obtain a model structure on the category of (relational) automata, where homotopy colimits are interesting, since they interact well with the language. However, the weak equivalences in this model structure are very rigid, since two weakly equivalent automata need to have the same transitions. On the other hand, the conditions on the states seem more reasonable, and can be compared to the notion of ``(co-)covering'' in~\cite[\S3.3.1]{sakarovitch2009elements}. This suggests that a localization as in~\cite{balchin2019bousfield} might give a model structure on automata which is not only interesting with regard to concatenation, but also with regard to the internal structure of the automata. This is currently under investigation.

\bibliographystyle{eptcsalpha}
\bibliography{b}

\appendix
\section{Omitted proofs}\label{sec:proofs}
\subsection{Proof of Lemma~\ref{lem:unique_lift}}
Suppose that $f$ has the unique right lifting property with respect to $i$. 
We just need to show that the right square as in the following commutative diagram has a filler:
\begin{center}
	\begin{tikzcd}
	c \arrow[r] \arrow[d, "i"'] & d\sqcup_c d\arrow[d, "\nabla_i"]\arrow[r, "t"] & \arrow[d, "f"]\\
	d \arrow[r, "\id_d"']\arrow[ur, "\iota_\epsilon"] & d \arrow[r, "w"'] \arrow[ur, "h"', dotted]& b
	\end{tikzcd}
\end{center}
The outer rectangle has a filler $h$, and we have that $f\circ h= w$. 
We still need to show that $h\circ \nabla_i= t$,
i.e. that $h\circ \nabla_i\circ\iota_\epsilon=t\circ\iota_\epsilon$ for $\epsilon=0,\,1$.
But this immediately follows from the fact that the filler of the outer rectangle is unique.
  
Conversely, consider the solid square on the left.
\begin{center}
	\begin{tikzcd}
	c \arrow[r, "t"] \arrow[d, "i"'] & a\arrow[d, "f"]\\
	d  \arrow[r, "w"'] \arrow[ur, "h_\epsilon"', dotted]& b
	\end{tikzcd}
   \qquad\qquad\qquad
   \begin{tikzcd}
	d\sqcup_c d\arrow[d, "\nabla_i"']\arrow[r, "({h_1}_\prime\, h_2)"] & a\arrow[d, "f"]\\
	d \arrow[r, "w"'] \arrow[ur, ""', dotted]& b
	\end{tikzcd}
\end{center}
with two fillers $h_1$ and $h_2$. 
The fact that $h_1=h_2$ follows from the fact that the square on the right has a filler.

\subsection{Proof of Proposition~\ref{prop:stable}}
Essentially by the small object argument~\cite[\S 2.1.2]{hovey2007model},
every element of $^\boxslash(I^\boxslash)$ is a retract of an $I$-cell extension,
which is a (countable) composite of pushouts of sums of elements of $I$.
So we prove by induction that $i\in\,^\boxslash(I^\boxslash)$ implies $\nabla_i\in\,^\boxslash(I^\boxslash)$. 
\begin{enumerate}
	\item (elements of $I$) By hypothesis, for every $i\in I$, $\nabla_i\in\,^\boxslash(I^\boxslash)$.
	\item (sums) Take a family $(i_\alpha)_\alpha$ of morphisms such that $\nabla_{i_\alpha}\in\,^\boxslash(I^\boxslash)$ for every $\alpha$.
	Since colimits commute with colimits, it is formal to show that \[\sqcup_\alpha \nabla_{i_\alpha}\cong \nabla_{\sqcup_\alpha i_\alpha}\]
	so $\nabla_{\sqcup_\alpha i_\alpha}\in \,^\boxslash(I^\boxslash)$.
	\item (pushouts) Consider the following pushout square on the left:
	\begin{center}
		\begin{tikzcd}
			a_1 \arrow[r,"f"]\arrow[d, "i_1"'] & a_2\arrow[d, "i_2"]\\
			b_1 \arrow[r, "g"'] & b_2
		\end{tikzcd}
     \qquad\qquad\qquad
     \begin{tikzcd}
			b_1\sqcup_{a_1} b_1 \arrow[r,"(g_\prime\,g)"]\arrow[d, "\nabla_{i_1}"'] & b_2\sqcup_{a_2} b_2 \arrow[d, "\nabla_{i_2}"]\\
			b_1 \arrow[r, "g"'] & b_2
		\end{tikzcd}
	\end{center}
	and suppose that $\nabla_{i_1}\in\,^\boxslash(I^\boxslash)$.
	We show that the square on the right is cocartesian, which proves that $\nabla_{i_2}\in\,^\boxslash(I^\boxslash)$. Take a cone:
	\begin{center}
		\begin{tikzcd}
			b_1\sqcup_{a_1} b_1 \arrow[r,"(g_\prime\,g)"]\arrow[d, "\nabla_{i_1}"'] & b_2\sqcup_{a_2} b_2 \arrow[d, "\nabla_{i_2}"] \arrow[ddr, bend left, "u"] \\
			b_1 \arrow[r, "g"'] \arrow[rrd, bend right, "v"'] & b_2 \\
			& & X
		\end{tikzcd}
	\end{center}
	This is equivalent to the data of $v$ and two morphisms $u_1,\,u_2$ such that 
	$u_1\circ i_2=u_2\circ i_2$ and $v=u_1\circ g=u_2\circ g$.
	But then $u_1=u_2$ because $b_2$ is the pushout of $i_1$ and $f$,
	so this is just the data of a morphism $b_2\xrightarrow{}X$,
	which concludes.
	\item (finite compositions) Take $i_1:a\xrightarrow{}b$ and $i_2:b\xrightarrow{}c$
	such that $\nabla_{i_1},\,\nabla_{i_2}\in\,^\boxslash(I^\boxslash)$. 
	It is formal to show that the following square is cocartesian:
	\begin{center}
		\begin{tikzcd}
			b\sqcup_a b \arrow[r, "({i_2}_\prime\,i_2)"] \arrow[d,"\nabla_{i_1}"'] & c\sqcup_a c\arrow[d, "({\id_c}_\prime\,\id_c)"]\\
			b \arrow[r] & c\sqcup_b c
		\end{tikzcd}
	\end{center}
	and that $\nabla_{i_2\circ i_1}$ is the composite 
	\[c\sqcup_a c \xrightarrow{} c\sqcup_b c \xrightarrow{\nabla_{i_2}} c\]
	\item (countable composition) Follows from the previous case and from commutation of colimits with colimits.
	\item (retracts) The codiagonal of a retract is a retract of the codiagonal.
	\end{enumerate}

\end{document}